\documentclass[11pt,reqno]{amsart}

\usepackage{amsmath}
\usepackage{amssymb}
\usepackage{amsfonts}
\usepackage{setspace}
\usepackage{mathrsfs}
\usepackage{version}

\newtheorem{theorem}{Theorem}[section]
\newtheorem{lemma}[theorem]{Lemma}

\newtheorem{corollary}[theorem]{Corollary}

\theoremstyle{definition}

\newcommand\op{\operatorname}
\renewcommand\c{\mathcal}
\renewcommand{\leq}{\leqslant}
\renewcommand{\geq}{\geqslant}
\newcommand\Sym{\mathcal{S}}

\def\F{\mathbf{F}}

\def\Z{\mathbf{Z}}

\def\N{\mathbf{N}}

\begin{document}

\title{A note on nonabelian Freiman-Ruzsa}



\author{Sean Eberhard}
\address{Mathematical Institute\\University of Oxford\\UK}
\email{eberhard@maths.ox.ac.uk}


\begin{abstract}
Recently Breuillard and Tointon~\cite{BT} showed that one reasonable formulation of the polynomial Freiman-Ruzsa conjecture fails for nonabelian groups. We improve and simplify their construction.
\end{abstract}

\maketitle

\section{Introduction}

A subset $A$ of some ambient group $G$ is called a \emph{$K$-approximate group} if $A^2\subset XA$ for some set $X$ of size at most $K$. A Freiman-type structure theorem for approximate groups was proved by Breuillard, Green, and Tao~\cite{BGT}.

\begin{theorem}[Breuillard-Green-Tao~\cite{BGT}]\label{BGT}
Let $A$ be a $K$-approximate group. Then there is a coset nilprogression $P\subset A^{100}$ of rank and step $O_K(1)$ such that $A$ is covered by $O_K(1)$ translates of $P$.
\end{theorem}

The former $O_K(1)$ is benign: it can be taken to be $O(K^2\log K)$, or even $O(\log K)$ at the expense of replacing $A^{100}$ in the theorem by $A^{O_K(1)}$. No earthly bounds are known for the latter $O_K(1)$, however, nor for the $O_K(1)$ in the following weaker statement.

\begin{corollary}\label{BGTweak}
Every $K$-approximate group $A$ is contained in $O_K(1)$ cosets of a group $\Gamma$ having a normal subgroup $H\subset A^{100}$ such that $\Gamma/H$ is nilpotent.
\end{corollary}

Recently Breuillard and Tointon~\cite{BT} gave an example which shows that the $O_K(1)$ here cannot be taken to be polynomial in $K$, even if we were to replace ``nilpotent'' with ``solvable'' and $A^{100}$ with, say, $A^{\exp\exp\exp K}$.

\begin{theorem}[Breuillard-Tointon~\cite{BT}]
Let $f:\N\to\N$ be arbitrary. Then for arbitrarily large $K$ there are finite $K$-approximate groups $A$ not covered by $K^{\frac1{200}\log\log\log\log K}$ cosets of any group $\Gamma$ having a finite normal subgroup $H\subset A^{f(K)}$ such that $\Gamma/H$ is solvable.
\end{theorem}

Here we improve and simplify their construction.

\begin{theorem}\label{BT}
For arbitrarily large $K$ there are finite $K$-approximate groups $A$ not covered by $K^{\log\log K}$ cosets of any group $\Gamma$ having a finite normal subgroup $H$ such that $\Gamma/H$ is solvable.
\end{theorem}

\section{Finite-by-solvable groups}

If $\c P$ and $\c Q$ are adjectives pertaining to groups (e.g., finite, nilpotent, or solvable), then a group $G$ is called $\c P$-by-$\c Q$ if it fits into an exact sequence $$1\to P \to G \to Q \to 1$$ in which $P$ is a $\c P$ group and $Q$ is a $\c Q$ group. (Some authors would call such a group $\c Q$-by-$\c P$. Perhaps they have the logical right of it, but we'll stick to the more common convention.)

Thus a group $\Gamma$ as in Theorem~\ref{BT} having a finite normal subgroup $H$ such that $\Gamma/H$ is solvable is called finite-by-solvable. Every finite-by-solvable group is solvable-by-finite, but the converse does not hold. For example consider the group $G = \Sym_n \ltimes \Z^n,$ where $\Sym_n$ acts on $\Z^n$ by permuting coordinates. $G$ is obviously solvable-by-finite, even abelian-by-finite, but the following lemma shows that $G$ is not finite-by-solvable.

\begin{lemma}\label{SymnZn}
If $G\leq\Sym_n$ is not solvable then $G\ltimes \Z^n$ is not finite-by-solvable.
\end{lemma}
\begin{proof}
Let $H$ be a normal subgroup of $G\ltimes\Z^n$ such that $(G\ltimes\Z^n)/H$ is solvable. Since $G/(H\cap G) \cong (GH)/H \leq (G\ltimes\Z^n)/H$ we must have ${G/(H\cap G)}$ solvable, and in particular $H\cap G\neq 1$. Moreover since $H$ is normal we must have $H\supset[H\cap G,\Z^n]$. But since $H\cap G\neq 1$ and the action of $\Sym_n$ on $\Z^n$ is faithful the subgroup $[H\cap G,\Z^n]$ of $\Z^n$ is nontrivial and therefore infinite, so $H$ must be infinite.
\end{proof}

We can now briefly sketch the proof of Theorem~\ref{BT}. We will take $G = \Sym_n\ltimes\Z^n$ and $A = \Sym_n\ltimes[-R,R]^n$, where $R\approx\infty$. Then $A$ is a $2^n$-approximate group. We will show that any $\Gamma\leq G$ which covers $A$ with only $2^{n\log n}$ cosets must cover so much of $\Sym_n$ that $\Gamma\cap\Sym_n$ is not solvable. $\Gamma$ will also have to cover a finite-index subgroup of $\Z^n$, so from Lemma~\ref{SymnZn} it will follow that $\Gamma$ is not finite-by-solvable.

\section{Solvable subgroups of the symmetric group}

The following theorem about solvable subgroups of symmetric group was proved by Dixon~\cite{dixon1}.

\begin{theorem}\label{solvable}
If $G\leq\Sym_n$ is solvable then $|G|\leq (2\cdot 3^{1/3})^{n-1}$.
\end{theorem}

For comparison, Dixon also proved that if $G\leq\Sym_n$ is abelian then $|G|\leq (3^{1/3})^n$, while if $G\leq\Sym_n$ is nilpotent then $|G|\leq 2^{n-1}$: see Dixon~\cite{dixon1,dixon2}. Each of these theorems is sharp up to a small multiplicative factor.

We reproduce the proof of Theorem~\ref{solvable} here just to keep this note reasonably self-contained.

\begin{proof}
Let $a=2\cdot 3^{1/3}$. We will prove $|G|\leq a^{n-1}$ by induction on $n$.

If $G$ is intransitive, say with orbit sizes $n_1,\dots,n_m$ for some $m>1$, then $G\leq G_1\times\cdots\times G_m$, where $G_i\leq\Sym_{n_i}$ is the restriction of $G$ to the $i$th orbit. Since each restriction $G_i$ is solvable, by induction $$|G|\leq a^{n_1 - 1}\cdots a^{n_m-1} \leq a^{n-1}.$$

Similarly if $G$ is transitive but imprimitive, say preserving a system of $r$ blocks with $1<r<n$, then the subgroup $H\leq G$ fixing each block is a solvable subgroup of $\Sym_{n/r}^r$, and $G/H$ is a solvable subgroup of $\Sym_r$, so by induction $$|G|=|G/H|\,|H|\leq a^{r-1} (a^{n/r-1})^r = a^{n-1}.$$

Hence we may assume $G$ is primitive. Let $A\leq G$ be a minimal normal subgroup. Then $A$ is characteristically simple (has no proper nontrivial characteristic subgroups), and any characteristically simple finite group is a direct product of isomorphic simple groups. Since $G$ is solvable this implies that $A\cong\F_p^k$ for some prime $p$ and some $k\geq 1$. Let $G_1$ be the stabilizer of a point in $G$. Since $G$ is primitive $G_1$ is maximal. Since $G$ is transitive $G_1$ contains no nontrivial normal subgroup of $G$, so $G=G_1A$. The centralizer $C(A)$ of $A$ in $G$ is normal in $G$, so $C(A)\cap G_1$ is normalised by both $A$ and $G_1$, hence by $G$. Thus $C(A)\cap G_1 = 1$. Thus $C(A)=A$ and $n = {[G:G_1]} = |A| = p^k$. It now follows that $G_1\cong G/C(A)$, which is isomorphic to a subgroup of $\op{Aut}(A) \cong \op{GL}_k(\F_p)$, so $$|G| \leq p^k|\op{GL}_k(\F_p)| = p^k(p^k-1)(p^k-p)\cdots(p^k-p^{k-1}).$$ Now check this is at most $a^{p^k-1}$ for all $p,k$.
\end{proof}

\section{Main argument}

Let $n\geq 100$, and let $G = \Sym_n \ltimes \Z^n$, where $\Sym_n$ acts on $\Z^n$ by permuting coordinates. For $r\geq 0$ let $A_r$ be the ball $$A_r = \Sym_n \ltimes [-r,r]^n.$$ Since $A_r^2 \subset A_{2r}$, $A_r$ is a $2^n$-approximate group for every $r$. Fix $R>2^{n\log n}$ and let $A=A_R$.

Suppose $A$ is contained in $2^{n\log n}$ cosets of $\Gamma\leq G$, i.e., $|A\Gamma/\Gamma|\leq 2^{n\log n}$. Then since $2^{n\log n} < R$ there is some integer $r\leq R$ such that $|A_{r+1}\Gamma/\Gamma|=|A_r\Gamma/\Gamma|$. Since $A_{s+1}=A_1A_s$ for all $s\geq 0$ this implies that $A_s\Gamma=A_r\Gamma$ for all $s\geq r$, and hence $A_r\Gamma=G$ and $A\Gamma=G$. Thus $\Gamma$ has index at most $2^{n\log n}$ in $G$.

In particular $\Gamma\cap\Sym_n$ has index at most $2^{n\log n}$ in $\Sym_n$, so by Theorem~\ref{solvable} and the calculation
\[
  2^{n\log n}(2\cdot 3^{1/3})^n < n!
\]
we know that $\Gamma\cap\Sym_n$ is not solvable. Also $\Gamma\cap\Z^n$ is a finite-index subgroup of $\Z^n$, so $\Gamma\cap\Z^n \geq d\Z^n$ for some integer $d>0$. Thus $\Gamma$ contains
\[
  (\Gamma\cap\Sym_n)\ltimes d\Z^n \cong (\Gamma\cap\Sym_n) \ltimes \Z^n,
\]
so by Lemma~\ref{SymnZn} the group $\Gamma$ cannot be finite-by-solvable.

\bibliography{nonabelianPFR}
\bibliographystyle{alpha} 
\end{document}